\documentclass[12pt]{amsart}

\usepackage{hyperref}

\usepackage[left=3cm,top= 3cm,bottom=3.5cm,right=3cm,portrait]{geometry}

 \usepackage{verbatim}
\usepackage{graphicx}
\usepackage{amsmath}            
\usepackage{amsfonts}           
\usepackage{amssymb}            
\usepackage{xcolor}             

\newtheorem{thm}{Theorem}[section]
 \newtheorem{cor}[thm]{Corollary}
 \newtheorem{lem}[thm]{Lemma}
 \newtheorem{prop}[thm]{Proposition}
 
 \theoremstyle{definition}
 
 \theoremstyle{remark}
 
 \newtheorem{ex}[thm]{Example}
 \numberwithin{equation}{section}

\newcommand{\om}{\ensuremath{\Omega}}
\newcommand{\cd}{\ensuremath{\mathbb{C}^d}}

\newcommand{\p}{\ensuremath{\partial}}
\newcommand{\eps}{\ensuremath{\varepsilon}}

\newcommand{\mam}{Monge-{A}mp\`ere measure }

\newcommand{\supp}{\text{supp} }

\begin{document}

\title{Weighted  $\theta$-Incomplete  Pluripotential Theory}
\address{Indiana University, Bloomington, IN 47405 USA}
\email{malan@indiana.edu}

\author{Muhammed Al\.I    ALAN}
\date{10 February 2009}
\maketitle
\begin{abstract} Weighted pluripotential theory is a rapidly   developing
area; and Callaghan \cite{Callaghan} recently introduced
$\theta$-incomplete polynomials in \cd\, for $d>1$. In this paper
we  combine these two theories by defining weighted
$\theta$-incomplete pluripotential theory. We define weighted
$\theta$-incomplete extremal functions  and obtain a
Siciak-Zahariuta type equality in terms of $\theta$-incomplete
polynomials. Finally we prove that the extremal functions can be
recovered using orthonormal polynomials and we demonstrate a result on strong asymptotics of Bergman functions in the spirit of \cite{BermanCn}.
\end{abstract}

\section{Introduction}
The theory of $\theta$-incomplete polynomials in \cd\, for $d>1$ was recently
developed by Callaghan \cite{Callaghan}. It has many applications in approximation
theory. He also defined interesting extremal functions in terms of
$\theta$-incomplete polynomials and related plurisubharmonic functions.

This paper has three goals. The first one is to further develop the
$\theta$-incomplete pluripotential theory of Callaghan. The second goal is to combine
this theory with weighted pluripotential theory and get a unified theory by defining
weighted $\theta$-incomplete pluripotential theory in \cd. If  $\theta=0$, we get
weighted pluripotential theory, and  for the weight $w=1$, we get
$\theta$-incomplete pluripotential theory. Finally we show that extremal functions in
these settings can be recovered asymptotically using orthonormal polynomials.

In this section  we recall some definitions and major results of weighted pluripotential theory and we recall Berman's paper
\cite{BermanCn} which is  a special case of weighted pluripotential theory. Our initial goal was to study  Berman's recent work
on globally defined weights within the framework of $\theta$-incomplete pluripotential theory. We were able to prove  many   results for admissible weights defined on closed subsets of \cd.

In the second section we recall some important results of $\theta$-incomplete
pluripotential theory. We improve a result of Callaghan and we extend a  result of
Bloom and Shiffman \cite{BloomShiffman} to the $\theta$-incomplete extremal
function $V_{K,\theta} $ associated to a compact set $K$ for $0\leq \theta <  1$.

In the third section we work on closed subsets of \cd.  We define the weighted $\theta$-incomplete extremal function
$V_{K,Q,\theta}$ for a closed set $K$ and an admissible weight function $w$ and we give various properties  of this
extremal function. We also show that  $V_{K,Q,\theta}$ can be obtained via taking the supremum of $\theta$-incomplete
polynomials whose weighted norm is less then or equal to 1 on $K$, generalizing the analogous result for $V_{K,\theta}$ (unweighted
case) from the previous section. In particular we state  analogous results in the case of global weights.

In the last section we recall the Bernstein-Markov property relating the sup norms and $L^2(\mu)$ norms of polynomials on
a compact set $K$ with measure $\mu$. We define a version of the Bernstein-Markov property for $\theta$-incomplete
polynomials in the weighted setting. Then we prove results on asymptotics of orthonormal polynomials to extremal functions in
the $\theta$-incomplete and weighted setting. Finally in Theorem \ref{SonTheorem}, we prove a result on strong
asymptotics of Bergman functions analogous to the main theorem in ~\cite{BermanCn}.

\subsection{Weighted Pluripotential Theory}\label{WeightedPT}

We give some basic definitions from  weighted pluripotential
theory. A good reference is Saff and Totik's book
\cite{Saff-Totik} for  $d=1$ and Thomas Bloom's Appendix B of
\cite{Saff-Totik} for $d>1$.

Let $K$ be a non-pluripolar closed subset of \cd.  An upper semicontinuous    function $w:K\to [0,\infty)$ is called an
admissible weight function on $K$ if
\begin{itemize}
\item[i)] the set $\{z \in K \, |\, w(z) >0 \}$  is not pluripolar and
\item[ii)] If $K$ is unbounded,   $|z|w(z)\to 0 $ as  $|z|\to \infty,\,z\in K$.
\end{itemize}
We define  $Q = Q_w = - \log w $, and we will use $Q$ and $w$ interchangeably.

The weighted pluricomplex extremal function of $K$ with respect to $Q$  is defined as
\begin{equation}
V_{K,Q}(z):= \sup\left\{ u(z) \,|\, u\in L, \, u\leq Q \text{ on } K\right\},
\end{equation}
where the Lelong class $L$ is defined as
\begin{equation}
 L:=\{ u\, |\, u \text{ is plurisubharmonic on } \cd, u(z) \leq \log^+ |z| +C \}.
\end{equation}

We recall that the  upper semicontinuous regularization of a function $v$ is defined by $v^\ast(z):=\limsup\limits_{w\to z}v(w)$ and it is well known that the upper semicontinuous regularization of $V_{K,Q}$ is plurisubharmonic and in $L^+$ where
$$L^+ := \{u \in L\, |\, \log^+ |z| +C \leq  u(z) \}.$$
By Lemma 2.3 of  Bloom's  Appendix B of \cite{Saff-Totik}, the support, $S_w$, of
$(dd^cV_{K,Q}^\ast)^d$ is a subset of $S_w^\ast:=\{z \in K\,|\, V_{K,Q}^\ast(z)
\geq Q(z)\}$.

Here $dd^c v= 2 i \p \bar{\p} v $ and $(dd^c v)^d$   is the complex Monge-{A}mp\`ere operator defined  by $(dd^c
v)^d=dd^c v \wedge \dots \wedge dd^c v$ for plurisubharmonic functions which are $\mathcal{C}^2$. For the cases
considered in this paper see \cite{Klimek, Demailly} for the details of the definition.

A set $E$ is called pluripolar if $E\subset \{z\in \cd \,|\,
u(z)=-\infty\}$ for some plurisubharmonic function $u$. If a
property holds everywhere except on a pluripolar set we will say
the property holds quasi everywhere.

 \subsection{A Special Case of Weighted Pluripotential Theory}

We recall some definitions from Berman's paper \cite{BermanCn}, where the weight is defined globally in \cd. Let $\phi$ be a
lower semicontinuous function, and $\phi(z) \geq (1+\eps)\log|z| \text{ for } z \gg 1$ for some fixed $\eps >0$. The
weighted extremal function is defined as
\begin{equation}
 V_{\phi}(z):=\sup\{ u(z)\,|\, u \in L \text{ and } u \leq \phi \text{ on } \cd\}.
\end{equation}
We define
\begin{eqnarray}
  S^\ast_\phi &:=&  \{z\in \cd \, |\, V_{\phi}^\ast(z)\geq \phi(z) \} \text{ and }\\
  S_\phi &:=&  \supp((dd^c V^\ast_{\phi})^d).
\end{eqnarray}

This is a special case of weighted pluripotential theory with $K=\cd$ and $Q=\phi$.
Hence $S_\phi \subset S^\ast_\phi$.

Berman \cite{BermanCn} studied the   case where the global weight $\phi \in
\mathcal{C}^{1,1}(\cd)$. In this case we define
\begin{eqnarray}
  D_\phi &=&  \{z\in \cd \, |\, V_{\phi}(z)= \phi(z) \},\\
  P      &=&  \{z\in \cd \, |\, dd^c {\phi}(z) \text{ exist and is positive}\}.
\end{eqnarray}
We remark that $D_\phi$ is a compact set and $S_\phi \subset D_\phi$. By Proposition 2.1 of \cite{BermanCn}, if $\phi\in
\mathcal{C}^{1,1}(\cd)$, then we have $V_{\phi}\in \mathcal{C}^{1,1}(\cd)$ and  $(dd^c V_{\phi})^d= (dd^c {\phi})^d$ on $D_\phi \cap P$ almost everywhere as $(d,d)$ forms with $L^\infty $ coefficients.

\begin{ex} Let $\phi(z)=|z|^2$.  Then we have
\begin{equation}\label{V0} V_{\phi}(z)=\left\{
\begin{array}{cl}
  |z|^2   & \text{ if } |z|\leq \frac{1}{\sqrt{2}},  \\
  \log|z| +\frac{1}{2}-\frac{1}{2}\log \frac{1}{2} &\text{ if } |z|\geq
  \frac{1}{\sqrt{2}}.
 \end{array}\right. \end{equation}\\
Clearly the plurisubharmonic function, $V$, on the right hand side is less then or equal to ~$\phi$, hence $V\leq V_{\phi}$.
On the other hand the support of the \mam of ~$V$  is the closed ball of radius $1/\sqrt{2}$ centered at the origin. Since
any competitor, $u$, for the extremal function is less then or equal to $|z|^2$ on this closed ball, by the domination
principle, (see Appendix B of \cite{Saff-Totik} or Theorem \ref{ThetaDominationPrinciple} below) $u$ is less then or equal
to ~$V$ on \cd. Therefore $V_\phi \leq V$ and hence equality holds.
\end{ex}

\section{$\theta$-Incomplete Pluripotential Theory}
We recall the basic notions
of $\theta$-incomplete pluripotential theory from \cite{Callaghan}. We fix $0\leq \theta
\leq 1$. A $\theta$-incomplete polynomial in \cd\,  is a
polynomial of the form
\begin{equation}\label{ThetaPolynomial}
P(z)=\sum\limits_{|\alpha|=\lceil N\theta\rceil}^N c_\alpha z^\alpha ,
\end{equation}
where $\lceil x \rceil$ is the least integer greater than or equal to $x$. Here we use the following multi-index notations. Let
$z=(z_1,\dots,z_d)\in \cd \text{ and } \alpha=(\alpha_1,\dots, \alpha_d)\in\mathbb{N}^d, \text{ then }$ %
$$z^\alpha = z_1^{\alpha_1} z_2^{\alpha_2}\dots  z_d^{\alpha_d} \quad \text{ and } \quad
 |\alpha| = \alpha_1+\dots +\alpha_d $$

The set of all $\theta$-incomplete polynomials of the form \eqref{ThetaPolynomial}
will be denoted by $\pi_{N,\theta}$. We remark that when $\theta=0$,
$\pi_{N,\theta}$  is the set of all polynomials of degree at most $N$; and when $\theta=1$,
$\pi_{N,\theta}$  is the set of homogenous polynomials of degree $N$.

Related classes of plurisubharmonic functions are defined as follows (See \cite{Callaghan} for details).
\begin{eqnarray}
 L_\theta     &=&\{ u\in L\, |\, u(z) \leq \theta \log |z| +C \text{ for } |z|<1 \}, \\
 L_\theta^+  &=&\{ u\in L_\theta\, |\,  \max(\theta \log |z|, \log |z|) +C \leq u(z) \text{ for all  } z\in
\cd\}.
\end{eqnarray}
We remark that if $P\in \pi_{N,\theta}$ then $\frac{1}{N}\log |P| \in L_\theta$. Another observation is if  $\theta_1 \geq
\theta_2$, then $ L_{\theta_2}\subset L_{\theta_1}$.

The next theorem gives a domination principle for $L_\theta$ classes.
\begin{thm}\label{ThetaDominationPrinciple}\cite[Theorem 3.15]{Callaghan} Let $0\leq\theta<1$. If $u\in L_\theta \text{
and } v\in L_\theta^+$ and if $u\leq v$ holds  almost everywhere with respect to $(dd^c v)^d$, then $u\leq v \text{ on }
\cd$.
\end{thm}
We remark that for $0<\theta<1$, we have $u(0)=v(0)=-\infty$ and the origin is a distinguished point as  it is charged by $(dd^c v)^d$.

Callaghan \cite{Callaghan}  defined the following extremal function for a   set $E
\subset \cd$:
\begin{equation}
 V_{E,\theta}(z):=sup\{ u(z): u \in L_\theta \text{ and } u \leq 0 \text{ on }E\}.
\end{equation}
 We will call it the $\theta$-incomplete extremal function of $E$.
The  upper semicontinuous  regularization,$V^\ast_{E,\theta}$, is in $ L_\theta^+ $
if $E$ is not pluripolar by  Lemma 3.7 of \cite{Callaghan}. Also if $K$ is a regular
compact set  in \cd, then $V^\ast_{K,\theta}=V_{K,\theta}$.  Hence it is continuous
 except at $z=0$. Here regular means the extremal function of $K$,
 $V_K:=V_{K,0}$  is continuous.  We  remark that  $(dd^c V_{E,\theta}^\ast)^d$ is supported in  $\bar{E}\bigcup\{0\}$.

According to   \cite{Callaghan} we have  the following result for compact sets $K$,
\begin{equation}\label{CallahanPhi} V_{K,\theta}= \log \Phi '_{K,\theta},\end{equation} where $$\Phi
'_{K,\theta}(z)=\sup\{|f(z)|^{1/N}:f\in \pi_{N,\theta}\text{ for some } N\geq 1
,\|f\|_K\leq1\}.$$ We  define the following functions for a compact set $K$. For
$N\geq 1$ we let
\begin{eqnarray}
\Phi_{K,\theta,N}(z)&=&\sup\{|f(z)|:f\in \pi_{N,\theta},\|f\|_K\leq1\} \text{
and}\\
\label{apple}\Phi_{K,\theta}&=&\sup\limits_{N}(\Phi_{K,\theta,N})^{1/N}.
\end{eqnarray}
The next proposition shows that the supremum in \eqref{apple} is actually a limit.

\begin{prop}\label{SupLim}With the above notation we have
$$\sup\limits_{N}\frac{1}{N}\log\Phi_{K,\theta,N}=\lim\limits_{N\to \infty}\frac{1}{N}\log\Phi_{K,\theta,N} \text{ and  } \Phi
'_{K,\theta}=\Phi_{K,\theta}.$$
Hence we have $\lim\limits_{N\to
\infty}\frac{1}{N}\log\Phi_{K,\theta,N}= V_{K,\theta}.$
\end{prop}
\begin{proof}
First of all we have $\Phi_{K,\theta,J}\,\Phi_{K,\theta,I}\leq\Phi_{K,\theta,J+I}$ for all integers $I, J\geq 0$. For if
$P(z)=\sum\limits^J_{|\alpha|=\lceil\theta J\rceil} a_\alpha z^\alpha$ and $Q(z)=\sum\limits^I_{\alpha =\lceil\theta I
\rceil} b_\alpha z^\alpha$, then $PQ(z)=\sum\limits^{J+I}_{\alpha=\lceil\theta J\rceil+\lceil\theta I\rceil} c_\alpha
z^\alpha$ is in $\pi_{J+I,\theta}$, since $\lceil\theta J\rceil+\lceil\theta I\rceil\geq\lceil\theta( J+I)\rceil$.

By taking logarithms  we get
\begin{equation}\label{LogInequality}
\log \Phi_{K,\theta,J}\,+\log \Phi_{K,\theta,I}\leq\log \Phi_{K,\Phi,J+I},
\end{equation}
so by Theorem 4.9.19 of \cite{BerensteinGay}, $\lim\limits_{N\to \infty}\frac{1}{N}\log\Phi_{K,\theta,N}$ exists and
equals $\sup\limits_{N}\frac{1}{N}\log\Phi_{K,\theta,N}$. Now by Callaghan's result \eqref{CallahanPhi} we get the last
equality $\Phi '_{K,\theta}=\Phi_{K,\theta}$.\end{proof}

In the next section we will extend this result to the weighted case.  This proposition
also fixes a gap in the proof of Theorem 8.2 in \cite{CallaghanThesis} and we will use
it in the proof of Theorem \ref{Asymptotic}.

The following theorem extends a result of Bloom and Shiffman \cite{BloomShiffman} to the $\theta$-incomplete case.

\begin{thm}\label{BSuniform} Let $K$ be a regular compact set in $\cd$. Then
 $$\frac{1}{N}\log\Phi_{K,\theta,N}\rightarrow V_{K,\theta}$$
 uniformly on compact subsets of $\cd\setminus\hat{K}_{\theta}$.
\end{thm}
Here  $\hat{K}_{\theta}$ is the $\theta-$incomplete hull of $K$  defined for a
compact set $K$ as
\begin{equation}
\hat{K}_{\theta}=\{z\in \cd\, |\, |p(z)|\leq \|p\|_K \text{ for all } p\in
\pi_{N,\theta} \text{ for  } N=0, 1,\dots \,\}.
\end{equation}
It is clear that for $\theta>0$, the origin always belongs to $\hat{K}_{\theta}$ for
any set $K$, so $\hat{K}_{\theta}$ is often larger then the usual polynomially convex
hull $\hat{K}:=\hat{K}_0$. It is also easy to see that $\hat{K}_{\theta}=\{z\in ~\cd\, |\,
V_{K,\theta}\leq0\}$.

\begin{proof} Let $E$ be a compact set in $\cd\setminus\hat{K}_{\theta}$.
First we want to show that there exists $N_0$ such that  $\Phi_{K,\theta,N}(z)>1
\text{ for all } N>N_0 \text{ for all } z\in E$.

We fix  $z_0\in E$ and  $\delta>0$ such that $V_{K,\theta}(z_0)=2\delta$. By the
above proposition we have $\lim\limits_{N\to\infty}\frac{1}{N}\log
\Phi_{K,\theta,N}(z_0)=2\delta$, so there exists an integer $N_{z_0}$ such that for
all $N\geq N_{z_0}$   we have $\frac{1}{N}\log \Phi_{K,\theta,N}(z_0)>\delta$. In
particular $\Phi_{K,\theta,N}(z_0)>1 \text{ for all } N>N_{z_0}$.

Since $\Phi_{K,\theta,N_{z_0}}$ is the supremum of continuous plurisubharmonic functions, it is lower semicontinuous.
Hence $U_{z_0}:=\{z\in \cd \,|\, \Phi_{K,\theta,N_{z_0}}(z)>1\}$   is open. Now we can cover ~$E$ by the sets
$U_{z_0}$, i.e., $E\subset \bigcup\limits_{z\in E} U_{z}$. There exists a finite subcover, $U_{z_1},..,U_{z_m}$, of $E$.
Hence taking $N_0$ to be the largest of $N_{z_1},\dots,N_{z_m}$, we can conclude that $\Phi_{K,\theta,N}(z) > 1$ for
all $z\in E$ and for all $N\geq N_0$.  Thus we have
$1\leq\Phi_{K,\theta,J}\leq\Phi_{K,\theta,J}\,\Phi_{K,\theta,I}\leq\Phi_{K,\theta,J+I}$ for all $I, J\geq N_0$ on $E$.

We follow \cite{BloomShiffman} to prove that the sequence  converges uniformly   on $E$. We will write $\psi_N =\frac{1}{N}\log\Phi_{K,\theta,N}$.
We note that  $ \psi_{Nk}\geq \psi_{N}$ for all $N\geq N_0$. We see this by $\psi_{Nk}= \frac{1}{Nk}\log\Phi_{K,\theta,Nk} \geq \frac{1}{Nk}\log(\Phi_{K,\theta,N})^k=\frac{k}{Nk}\log\Phi_{K,\theta,N}=\psi_{N}$ for $N\geq N_0$.
From \eqref{LogInequality}, we have $Nk\psi_{Nk}+j\psi_j\leq (Nk+j)\psi_{Nk+j}$ for
$N, k \geq 1, \, j\geq 0$. Since $\psi_j>0$ on $E$ for $j>N_0$, using $\psi_{Nk}\geq \psi_{N}$  for such $j$  we get
\begin{equation}\label{BS2}\psi_{Nk+j}\geq
\frac{Nk}{Nk+j}\psi_N+\frac{j}{Nk+j}\psi_j\geq \frac{Nk}{Nk+j}\psi_N.\end{equation} Let   $\eps>0$. For each $a\in E$
we can choose $N_a>N_0$ large so that $V_{K,\theta}(a)-\psi_{N_a}(a)<~\eps$ and $\frac{V_{K,\theta}(a)}{N_a}<\eps$,
and then we can find an open neighborhood $U_a$ of $a$ such that $|V_{K,\theta}(z)-V_{K,\theta}(a)|<\eps, \,
\psi_{N_a}(z)> \psi_{N_a}(a)-\eps, \text{ and } \frac{V_{K,\theta}(z)}{N_a}<\eps \text{ for } z\in U_a$. This is possible by
the facts that  regularity of $K$ implies the continuity of $V_{K,\theta}$ and that $\psi_{N_a}$ is lower semicontinuous.

Now we find a finite number of points $a_1, \dots a_M$ in $E$ such that the open
sets $ U_{a_1},\dots, U_{a_M}$ cover $E$. We choose $N_1
=\max\limits_{a_1,\dots,a_M}(N_{a_i}^2+N_{a_i})$. Now for each $a_i$ if $N\geq
(N_{a_i}^2+N_{a_i})$, we write $N=N_{a_i} (k-1)+j$, where $k\geq N_{a_i}, \text{
and } N_{a_i}\leq j\leq 2N_{a_i}$. By Proposition \ref{SupLim} and \eqref{BS2} we
get
$$0\leq V_{K,\theta}-\psi_N\leq
V_{K,\theta}-\frac{N_{a_i}(k-1)}{N_{a_i}(k-1)+j}\psi_{N_{a_i}}\leq V_{K,\theta}-\frac{N_{a_i}}{N_{a_i}
+2}\psi_{N_{a_i}}\leq V_{K,\theta}-\psi_{N_{a_i}}+\frac{2}{N_{a_i}+2}V_{K,\theta}.$$

Let $z\in E$, then  $z\in U_{a_i}$ for some $a_i$, hence for all $N\geq N_1$  we have
\begin{eqnarray}
\nonumber 0 & \leq & V_{K,\theta}(z)-\psi_N(z)<V_{K,\theta}(z)-\psi_{N_{a_i}}(z)+2\eps\\
\nonumber   &  =   &[V_{K,\theta}(z)- V_{K,\theta}(a_i)]+ [ V_{K,\theta}(a_i)-\psi_{N_{a_i}}(a_i)]+[\psi_{N_{a_i}}(a_i)-\psi_{N_{a_i}}(z)]+2\eps\\
\nonumber   &\leq &5\eps.
\end{eqnarray}
Thus we have the desired uniform convergence on $E$.\end{proof}

\section{ Weighted $\theta$-Incomplete Pluripotential
Theory}\label{ThetaWeihgtedIncPP}

In this section we define  and develop two weighted versions of $\theta $-incomplete pluripotential theory. The first one is the  $\theta
$-incomplete version of the weighted pluripotential theory in closed subsets of \cd\,  and the second  one is the  $\theta
$-incomplete version of the special case of weighted pluripotential theory studied in \cite{BermanCn}. As in the  $\theta=0$
case the second version is a special case of the first.
\subsection{Weighted $\theta$-Incomplete Pluripotential Theory with Weight Defined
on Closed Sets }

Let $K$ be a closed set in \cd\, and $w$ be an admissible weight on $K$ as defined
in Subsection  \ref{WeightedPT}. Then we define
\begin{equation}
V_{K,Q,\theta}(z):= \sup\left\{ u(z)\, |\, u\in {L_\theta}, \, u\leq Q \text{ on }
K\right\}.
\end{equation}

We remark that $V_{K,Q, \theta_1}\leq V_{K,Q, \theta_2}$ if $\theta_1 > \theta_2$. The $\theta=0$ case gives  the
classical weighted pluripotential theory. Following Siciak \cite{SiciakExtremal1},  it can be shown that
$V_{K,Q,\theta}=V_{K,Q,\theta}^\ast$, so that $V_{K,Q,\theta}$ is continuous on $\cd\setminus \{0\}$, for $K$ locally
regular and $Q$ continuous. Here ~$K$ locally regular means for all $a\in K , \text{ we have }  K\cap \overline{B(a,r)}$ is
regular for all $r>0$, where $B(a,r):=\{z\in\cd\,|\,|z-a|<r\}$.

Comparing the defining families we get  the following obvious inequalities.
\begin{prop}\label{IlkInequality}
Let $K_1\subset K_2$ and let $w$ be a function defined on $K_2$
which is an admissible weight on both $K_1 \text{ and } K_2$. Then
$V_{K_1,Q,\theta}\leq V_{K_2,Q,\theta}$.
\end{prop}

Using (ii) in the definition of admissibility from section 1.1, we  show that $V_{K,Q,\theta}$
coincides with the weighted $\theta$-incomplete extremal function of a  compact subset of $K$.
\begin{lem}\label{COMpactWeight}
If $K$ is unbounded then $V^\ast_{K_\rho,Q,\theta}=V^\ast_{K,Q,\theta}$, for
some $\rho>0$ where $K_\rho= \{z\in K\,|\, |z|\leq \rho\}$.
\end{lem}
\begin{proof}
Since $V^\ast_{K_\rho,Q,\theta} \in L$, there exists $C$ and $\rho$ such that
$$V^\ast_{K_\rho,Q,\theta}(z)\leq \log|z|+C \text{ for  } |z| > \rho . $$
Now by the second condition of admissibility  we may choose $\rho$ large enough that
$$Q(z)-\log|z|\geq C+1 \text{ for  } z\in K\setminus K_\rho. $$
If $u\in L_\theta$ and $u\leq Q  \text{ on  } K_\rho$, so that $u\leq V^\ast_{K_\rho,Q,\theta}$, by the above
inequalities we get $u\leq Q \text{ on } K$. Hence we get $V^\ast_{K_\rho,Q,\theta}\leq V^\ast_{K,Q,\theta}$. The other
inequality is given by Proposition ~\ref{IlkInequality}, which gives the   equality.
\end{proof}

\begin{prop}
Let $K$ be a closed subset of \cd\,  and let $w$ be an admissible weight function on
$K$ then $V^\ast_{K,Q,\theta} \in L_\theta^+$.
\end{prop}

\begin{proof}
The case  $\theta =0$ is the classical case and is well known.  For $0< \theta \leq 1$
we will follow the proof of Lemma 3.7 of \cite{Callaghan}.

Since $V^\ast_{K,Q,\theta}\leq V^\ast_{K,Q}  \text{ and } V^\ast_{K,Q} \in L^+$,
we have $V^\ast_{K,Q,\theta} \in L$.

Next we show that $V^\ast_{K,Q,\theta} \in L_\theta$. Let $M:=\sup\limits_{z\in
B(0,1)} V^\ast_{K,Q,\theta}(z)$ and $u$ be in the defining class for
$V_{K,Q,\theta}$. Then $\frac{1}{\theta}(u-M)\leq 0 \text{ on } B(0,1)$. Hence it is
a competitor for the pluricomplex Green function of the unit ball $B(0,1)$  with
logarithmic pole at the origin. The pluricomplex Green function of a bounded domain
\om\,  with logarithmic pole at $a\in \om$ is defined by $$g_\om(z,a):=\sup\{u(z)\,
|\, u \text{ plurisubharmonic on } \om ,\,  u\leq 0 \text{ and } u(z)-\log |z-a | \leq C
\text{ as }z\to a\},$$ and $g_{B(0,1)}(z,0)=\log |z|$. Hence
$\frac{1}{\theta}(u-M)\leq \log|z| $ on the unit ball. Since $u$ is arbitrary we get
$V^\ast_{K,Q,\theta}(z)\leq \theta \log|z| +M \text{ on } B(0,1)$. Thus
$V^\ast_{K,Q,\theta} \in L_\theta$.

By  Lemma \ref{COMpactWeight} we may assume $K\subset B(0,R)$ for some $R$.
Let  $A:=\sup\limits_{z\in B(0,R)} (\theta\log|z|-Q(z) )$, then $u(z)=\max(\theta
\log|z|, \log|z|)-A$ is a competitor for the extremal function $V_{K,Q,\theta} $
and $u\in L_\theta^+$, hence $V^\ast_{K,Q,\theta} \in L_\theta^+$.
\end{proof}

We define the following sets:
\begin{eqnarray}
  S^\ast_{K,Q, \theta}  &:=&  \{z\in K\, |\, V^\ast_{K,Q, \theta}(z)\geq Q(z) \} \text{ and }\\
  S_{K,Q,\theta}  &:=&  \supp((dd^c V^\ast_{K,Q, \theta})^d).
\end{eqnarray}


\begin{lem}\label{SupS}
Let $K$ be closed in \cd\,  and let $w$ be an admissible weight on $K$. Then $
S_{K,Q, \theta} \subset {S^\ast_{K,Q, \theta}}\bigcup\{0\}$ if $0<\theta \leq 1$
and $S_{K,Q, \theta} \subset {S^\ast_{K,Q, \theta}}$ if $\theta =0$.
\end{lem}

\begin{proof} The classical case, i.e. when $\theta=0$,  is  Lemma
2.3 of Appendix B of \cite{Saff-Totik}. Therefore we assume
$0<\theta \leq 1$. Let $z_0$ be a point in $K\setminus\{0\}$  such
that $V^\ast_{K,Q,\theta}(z_0)< Q(z_0)-\eps$ for some positive
$\eps$. We will show that $V_{K,Q,\theta}^\ast$ is maximal in a
neighborhood of $z_0$, i.e $(dd^c V_{K,Q,\theta}^\ast)^d=0 $
there.

Since $Q$ is lower semicontinuous we have $\{z\in K\,| \, Q(z)>Q(z_0)-\eps/2\}$ is
open relative to $K$. Similarly we have $\{z\in \cd \,|\, V_{K,Q,\theta}^\ast(z)<
V_{K,Q,\theta}^\ast(z_0)+\eps/2\}$ is open. Thus we may find  a ball of radius $r$
around $z_0$ such that  $\sup\limits_{z\in B(z_0,r)}V_{K,Q,\theta}^\ast(z)<\inf
\limits_{z\in B(z_0,r)\cap K}Q(z)$ and $0 \not\in B(z_0,r)$.

By Theorem 1.3 of Appendix B in \cite{Saff-Totik}, we can find a plurisubharmonic
function $u$ with  $u\geq V_{K,Q,\theta}^\ast$ on $B(z_0,r)$,  $u
 =V_{K,Q,\theta}^\ast$ on $\cd\setminus B(z_0,r)$, and $u$
 maximal on $B(z_0,r)$. Then $u\leq V_{K,Q,\theta}^\ast$
 because $u(z)\leq \sup\limits_{z\in B(z_0,r)}V_{K,Q,\theta}^\ast(z)<\inf
\limits_{z\in B(z_0,r)\cap K}Q(z)$ for all $z\in B(z_0,r)$. Since $B(z_0,r)\cap\{0\}$ we have $u\in L_\theta$. Hence
$u\equiv V_{K,Q,\theta}^\ast$. Therefore we get $V_{K,Q,\theta}^\ast$ is maximal in a neighborhood  of $z_0$. Hence
$z_0$ is not in  $S_{K,Q, \theta} $.
\end{proof}

A special case of this is when the admissible weights are globally
defined. Let $\phi~:~\cd \to \mathbb{R}$ be an admissible weight
function.  Generalizing the case of \cite{BermanCn} we define
weighted $\theta$-incomplete extremal functions by
\begin{equation}
 V_{\phi, \theta}(z)= \sup\{u(z)\,|\, u\in L_\theta \text{ and } u \leq \phi  \} \text{ for } 0\leq \theta\leq 1 .
\end{equation}
Observe that $V_{\phi, \theta}^\ast= V_{\phi, \theta}$ if $\phi $
is continuous, for in this case $V_{\phi, \theta}^\ast\leq \phi $
on \cd\, so that $V_{\phi, \theta}^\ast\leq  V_{\phi, \theta}$. We
also remark that $\theta =0 $  gives $V_{\phi, 0}= V_{\phi}$ and
$V_{\phi, \theta_1}\leq V_{\phi, \theta_2}$ if $\theta_1 >
\theta_2$ since $L_{\theta_1}\subset L_{\theta_2}$.

We define the following sets:
\begin{eqnarray}
  D_{\phi, \theta}  &:=&  \{z\in \cd \, |\, V^\ast_{\phi, \theta}(z)\geq \phi(z) \} \text{ and }\\
  S_{\phi, \theta}  &:=&  \supp((dd^c V^\ast_{\phi, \theta})^d).
\end{eqnarray}

If $\theta =0$, we will write  $D_{\phi, 0}= D_{\phi}$ and $S_{\phi, 0}= S_{\phi}$.
If $\phi $   is continuous then  $V_{\phi, \theta}$ is continuous and  we have
$$  D_{\phi, \theta}  =  \{z\in \cd \, |\,V_{\phi, \theta}(z)=\phi(z) \}.\quad\quad$$

If $\phi$ is a globally defined admissible weight function  then we define
$K:=D_{\phi, \theta}$ and $Q:=\phi |_{K}$. Clearly $V_{\phi,\theta}^\ast\leq Q$
quasi everywhere  in $K$ so $V_{\phi,\theta}^\ast \leq V_{K,Q,\theta}^\ast$.

Conversely, on $K,\,  V_{K,Q,\theta}\leq Q=\phi = V_{\phi,\theta}
$ quasi everywhere. Since $(dd^c V^\ast_{\phi,\theta})^d $ is
supported on $K\bigcup \{0\}$, by Theorem
\ref{ThetaDominationPrinciple} we have $V_{K,Q,\theta}^\ast \leq
V_{\phi,\theta}^\ast$. Hence $V_{K,Q,\theta}^\ast =
V_{\phi,\theta}^\ast$. This shows that we may reduce the global
weighted situation to the compact case by considering the sets
$D_{\phi,\theta}$.

As a consequence of the above definitions,  Lemma \ref{SupS} and earlier results of this section we have the following
corollary.

\begin{cor}\label{SupportPhi}
Let $\phi $ be a globally defined admissible weight, then we have
\begin{enumerate}
  \item[i)]  $S_{\phi, \theta}=\supp((dd^c V^\ast_{\phi, \theta})^d) \subset D_{\phi, \theta}
\bigcup \{0\}$ if $\theta>0$, and for $\theta =0, \\ \supp((dd^c V^\ast_{\phi})^d)
\subset~ D_{\phi} $,
  \item[ii)] $D_{\phi, 1}\subset D_{\phi, \theta_1}\subset
  D_{\phi, \theta_2}\subset D_{\phi, 0} =D_\phi  \text{ where } \theta_1 >
  \theta_2$,
  \item[iii)] $ V_{\phi, \theta}$ is in $L^+_\theta$ for  $0\leq \theta \leq
1$,
\item[iv)]
if $u\in L_\theta \text{ and } u \leq \phi \text{ on } D_{\phi,\theta}$  then $u \leq
V_{\phi,\theta}$.
\end{enumerate}
\end{cor}

The next lemma shows the monotonicity of the extremal functions under increasing
and  decreasing $\theta$.

\begin{lem}\label{Increasing}
Let $K \subset \cd$ be a closed set and let $w$ be an admissible weight on $K$. For
$0\leq \theta_0<1$, as $\theta \searrow \theta_0$ we have
$V^\ast_{K,Q,\theta}$ increases to $V^\ast_{K,Q,\theta_0}$ quasi everywhere. If
$\theta \nearrow \theta_0$ we have $V^\ast_{K,Q,\theta}$ decreases  to
$V^\ast_{K,Q,\theta_0}$.
\end{lem}
\begin{proof}The last statement is clear, thus we consider $\theta \searrow
\theta_0$. Clearly we have monotonicity of the $V^\ast_{K,Q,\theta}$. Since
$V^\ast_{K,Q,\theta}$ are bounded above by  $V^\ast_{K,Q,\theta_0}$, we have
 $V^\ast_{K,Q,\theta}$ increases to a function, $v$, whose upper semicontinuous
 regularization $v^\ast$ is plurisubharmonic and again bounded above by
 $V^\ast_{K,Q,\theta_0}$.

Since  $V^\ast_{K,Q,\theta} \in L_\theta^+$ we have $V^\ast_{K,Q,\theta}(z) \geq \max(\theta \log|z|,
\log|z|)+M_\theta$ where $M_\theta $ is a constant depending on $\theta$. As $\theta \searrow \theta_0$ we get
$v^\ast\in L_{\theta_0}^+$ since $v^\ast\leq V^\ast_{K,Q,\theta_0}$. Also by monotonicity we get
$(dd^cV^\ast_{K,Q,\theta})^d\to(dd^c v^\ast)^d$ weak-*.

We will write  $S:=\supp (dd^c v^\ast)^d\setminus\{0\}$ and  $S':=\{z\in
K\,|\,v^\ast(z)\geq Q(z)\}$. By the lower semicontinuity of $Q$, and upper
semicontinuity of $v^\ast $, we have $S'$ is closed. Next we will show that  $v^\ast
\geq Q $ on $S$ by showing that $S \subset S'$.

Since $(dd^cV^\ast_{K,Q,\theta})^d\to(dd^c v^\ast)^d$ we have $S \subset
\overline{\bigcup\limits_{\theta >\theta_0}S_{K,Q,\theta_0}}\setminus \{0\}$. By
Proposition \ref{SupS}, we have $\bigcup\limits_{\theta
>\theta_0}S_{K,Q,\theta}\setminus \{0\}\subset \bigcup\limits_{\theta >\theta_0}S^\ast_{K,Q,\theta}\setminus \{0\}\subset \{z\in K\,|\, v(z)\geq
Q(z)\}\subset S'$. Since $S'$ is closed, we get $\overline{\bigcup\limits_{\theta
>\theta_0}S_{K,Q,\theta_0}}\setminus \{0\}\subset S'$. Therefore $S \subset S'$.
Since $V^\ast_{K,Q,\theta_0}\leq Q$ quasi everywhere on $K$ and $(dd^c v^\ast)^d$ does not charge pluripolar sets
except the origin,  we have $V^\ast_{K,Q,\theta_0}\leq v^\ast $ almost everywhere with respect to $(dd^c v^\ast)^d$ on
the support of $(dd^c v^\ast)^d$. Here we recall that if $\theta>0$ then $V^\ast_{K,Q,\theta_0}(0)=v^\ast(0)=-\infty$.
Therefore by the domination principle (Theorem ~\ref{ThetaDominationPrinciple}) we get $V^\ast_{K,Q,\theta_0}\leq
v^\ast $ on \cd, so that $V^\ast_{K,Q,\theta_0}=~v^\ast$.
\end{proof}
\begin{cor}Let $\phi$ be a globally defined admissible weight.  Let $0\leq\theta_0<1$, as $\theta \searrow \theta_0$ we have
$V_{\phi,\theta}^\ast$ increases to $V_{\phi,\theta_0}^\ast$ quasi everywhere,
and if $\theta \nearrow \theta_0$  we have $V_{\phi,\theta}^\ast$ decreases to
$V_{\phi,\theta_0}^\ast$.
\end{cor}

The following example illustrates the above corollary.

\begin{ex} Let $\phi(z)=|z|^2$.  Then we have for $0<\theta <1$
$$V_{\phi, \theta}(z)=\left\{
\begin{array}{cl}
\theta \log|z|+\frac{\theta}{2}-\frac{\theta}{2}\log \frac{\theta}{2} & \text{ if } |z| < \sqrt{\frac{\theta}{2}},\\
  |z|^2   & \text{ if } \sqrt{\frac{\theta}{2}}\leq |z|\leq \sqrt{\frac{1}{{2}}},  \\
  \log|z| +\frac{1}{2}-\frac{1}{2}\log \frac{1}{2} &\text{ if } |z|\geq  \sqrt{\frac{1}{{2}}}.
 \end{array}\right. $$
If $\theta=1$ we get
$$V_{\phi, \theta}(z)=V_{\phi, 1}(z)=\log|z| +\frac{1}{2}-\frac{1}{2}\log
\frac{1}{2}.$$ We had given $V_{\phi,0}$ earlier in \eqref{V0}.

Note that
$D_{\phi,\theta}=\overline{B(0,\frac{1}{\sqrt{2}})}\setminus
B(0,\sqrt{\frac{\theta}{2}})$  which increases to
$\overline{B(0,\frac{1}{\sqrt{2}})}\setminus\{0\}$ as $\theta$
decreases to 0.
\end{ex}

We define the following notions. Let $K \subset \cd$  be compact  and $w$ be an
admissible weight on $K$. We define
\begin{equation}
\Phi_{K,Q,\theta}^N(z):=\sup\{|P(z)|^{1/N}\,|\, \|w^N P_N\|_{K} \leq 1\text{
where }P_N\in \pi_{N,\theta}\}
\end{equation}
and
\begin{equation}
\Phi_{K,Q,\theta}:=\sup_{N}\{\Phi_{K,Q,\theta}^N\}=\lim_{N\to\infty}\Phi_{K,Q,\theta}^N.
\end{equation}

We can see that the supremum is actually a limit by following  the proof of Proposition
~\ref{SupLim}.

\begin{thm}\label{PolynomialGreen}Let $0\leq \theta<1$. Let $K \subset \cd$ be a compact set
and $w$ be a continuous admissible weight on
$K$. Then $V_{K,Q,\theta}=\log{\Phi_{K,Q,\theta}}$.
\end{thm}
\begin{proof}
Let $P_N\in \pi_{N,\theta} $ satisfying $\|w^N P_N\|_{K} \leq~ 1$.
Then we have
$$\frac{1}{N}\log|P_N(z)| \leq Q(z) \text{ on } K.$$ Hence we get
\begin{equation}\label{proof1st}
\log \Phi_{K,Q,\theta}\leq V_{K,Q,\theta}.
\end{equation}

The rest of the proof  essentially follows the proof of Callaghan \cite{Callaghan}.  We will modify the last step using a  result
of Brelot-Cartan  instead of Hartog's lemma.

We fix $\eps >0$ such that $\theta+\eps<1$. Let $u\in L_{\theta+\eps} \text{ and }
u \leq Q \text{ on } K$. By Theorem 2.9 of Appendix B of \cite{Saff-Totik}, we have
$$u(z)=\lim_{j\to \infty} \frac{1}{N_j}\max_{1\leq k \leq t_j}\log|P_{k,j}(z)|,$$
where the sequence is decreasing and each $P_{k,j}$ is a polynomial of degree at most
$N_j$. Here  $t_j$ is a finite number depending on  $j$.

As in \cite{Callaghan}  we write
$$P_{k,j}(z):=\sum_{|\alpha|=0}^{N_j} c_{\alpha,k,j}z^\alpha$$ and
$$P_{k,j}'(z):=\sum_{|\alpha|=0}^{\lfloor N_j \theta\rfloor}c_{\alpha,k,j}z^\alpha,$$
where $\lfloor x\rfloor$ is the largest integer less then or equal to $x$.

We remark that $P_{k,j}-P_{k,j}'$ is a $\theta$-incomplete polynomial. By Callaghan's asymptotic estimates  we get
 $$u(z)=\lim_{j\to \infty} \frac{1}{N_j}\max_{1\leq k \leq
t_j}\log|P_{k,j}(z)-P_{k,j}'(z)|$$pointwise on \cd.

By Theorem 3.4.3 c) of \cite{RansfordPotentialTheoryInTheComplexPlane}, for $\eps_1>0$, there exists $j_1$ such that
for $j\geq j_1$ we have
 $$ \frac{1}{N_j}\max_{1\leq k \leq t_j}\log|P_{k,j}(z)-P_{k,j}'(z)| \leq Q +\eps_1\text{ on } K,$$since $Q$ is continuous.
Now we have
$$u(z)=\lim_{j\to \infty} \frac{1}{N_j}\max_{1\leq k \leq
t_j}\log|P_{k,j}(z)-P_{k,j}'(z)|\leq \log \Phi_{K,Q,\theta}(z)+\eps_1$$for any
$\eps_1$  and therefore $u(z)\leq \log \Phi_{K,Q,\theta}(z)$. Hence we get
$$
V_{K,Q,\theta+\eps}(z)\leq \log \Phi_{K,Q,\theta}(z).
$$
By Lemma \ref{Increasing}, as $\eps\to 0 $ we get
\begin{equation}\label{Elma}
V_{K,Q,\theta}(z)\leq \log \Phi_{K,Q,\theta}(z).
\end{equation}
Combining \eqref{Elma}  with \eqref{proof1st} we get the desired  result.
\end{proof}

Note that if $\theta =0$, we recover
\begin{equation}\label{REcovertheta=0}
V_{K,Q}=\log \Phi_{K,Q} \text{ where }  \Phi_{K,Q}:= \Phi_{K,Q,0}
\end{equation}

\begin{cor}Let $0\leq \theta<1$. Let $\phi$ be a globally defined continuous admissible weight, then we have
$V_{\phi,\theta}=\log{\Phi_{\phi,\theta}}$, where
\begin{equation}
\Phi_{\phi,\theta}^N(z):=\sup\{|P(z)|^{1/N}\,|\,
\|e^{-N\phi}P_N\|_{D_{\phi,\theta}} \leq 1\text{ where }P_N\in \pi_{N,\theta}\}
\end{equation} and
\begin{equation}
\Phi_{\phi,\theta}:=\sup_{N}\{\Phi_{\phi,\theta}^N\}.
\end{equation}
\end{cor}
\begin{cor}Let $0\leq \theta<1$. Let $\phi$ be a globally defined continuous admissible weight, then we have
$V_{\phi,\theta}=\log{{\widetilde{\Phi}}_{\phi,\theta}}$, where
\begin{equation}
{\widetilde{\Phi}}_{\phi,\theta}^{N}(z):=\sup\{|P(z)|^{1/N}\,|\,
\|e^{-N\phi}P_N\|_{\cd} \leq 1\text{ where }P_N\in \pi_{N,\theta}\}
\end{equation} and
\begin{equation}
\widetilde{\Phi}_{{\phi,\theta}}:=\sup_{N}\{{\widetilde{\Phi}}_{\phi,\theta}^{N}\}.
\end{equation}
\end{cor}
\begin{proof}
It is sufficient to show that for any  $ P_N\in \pi_{N,\theta},\, \|e^{-N\phi}P_N\|_{\cd} \leq 1$ { if and only if  }
$\|e^{-N\phi} P_N\|_{{D_{\phi,\theta}}} \leq 1$. The "only if" direction is trivial. For the other direction let $ P_N\in
\pi_{N,\theta}$  and $\|e^{-N\phi}P_N\|_{{D_{\phi,\theta}}} \leq 1$. We will show that $\|e^{-N\phi}P_N\|_{\cd} \leq
1$.  We have $e^{-N\phi(z)}P_N(z) \leq 1$ for $z \in {D_{\phi,\theta}}$ so we get $\frac{1}{N}\log|P_N(z)| \leq \phi(z)
\text{ on } {D_{\phi,\theta}}$. Hence it is a competitor for the extremal function $V_{\phi,\theta}$, and we have
$\frac{1}{N}\log|P_N(z)| \leq V_{\phi,\theta}(z) \leq~\phi(z) \text{ for all } z \in~\cd$. Therefore we get
$e^{-N\phi(z)}P_N(z) \leq 1 \text{ for all } z \in \cd$.
\end{proof}

\section{Asymptotics}
Let $K$ be a compact set in \cd\,  and $\mu$ be a Borel probability measure whose support is in $K$. We say that the pair
$(K,\mu)$ satisfies a \textbf{Bernstein-Markov property} if for any $\eps>0$ there exists $C>0$ such that
\begin{equation}\label{BM-General}
\|P\|_K\leq C e^{\eps N}\|P\|_{L^2(\mu)}
\end{equation}
holds  for all polynomials of degree at most $N$. Equivalently, there exists $M_N$
with  $(M_N)^{\frac{1}{N}}\to~ 1  \text{ as } N\to \infty$ such that the following
inequality holds for all polynomials of degree at most $N$:
\begin{equation}\label{BM-MnForm}
  \|P\|_K \leq M_N \|P\|_{L^2(\mu)}  .
\end{equation}

We remark that if $K$ is a regular compact set then $(K,(dd^c V_K)^d)$ satisfies the
Bernstein-Markov property. See \cite{Zeriahi} for details.

We fix $0\leq \theta \leq 1$. If  these inequalities are satisfied for all $P\in \pi_{N,\theta}$ for all $N\geq 0$, then we
say the pair  $(K,\mu)$ satisfies a \textbf{Bernstein-Markov property for $\theta$-incomplete polynomials}.

Let $\mu$ be a measure such that  $(K,\mu)$ satisfies the Bernstein-Markov property for $\theta$-incomplete
polynomials. Let $\{P_j\}$ be an orthonormal basis of $\pi_{N,\theta}$ with respect to the inner product $\langle f,g
\rangle :=\int f\bar{g}\,d\mu$. We define the Bergman function
$K_{N,\theta}(z,w):=\sum^{d(N,\theta)}_{j=1}P_j(z)\overline{P_j(w)}$, where $d(N,\theta)$ is the dimension of
$\pi_{N,\theta}$.

The following two lemmas are generalizations of results of Bloom and Shiffman
\cite{BloomShiffman}.

\begin{lem}  If $(K,\mu) $ satisfies the Bernstein-Markov property for $\theta$-incomplete polynomials, then for
all $\epsilon >0$, there exists
$C>0$ such that
\begin{equation}\label{BSL1}\frac{(\Phi_{K,\theta,N}(z))^2}{d(N,\theta)} \leq
{K_{N,\theta}(z,z)} \leq C e^{\epsilon N}{(\Phi_{K,\theta,N}(z))^2}
d(N,\theta)\end{equation}  for all $z\in \cd$.
\end{lem}
\begin{proof}
To show the first inequality we take $P\in \pi_{N,\theta}$ and $\|P\|_K\leq 1$.
Then we have
\begin{eqnarray}
 \nonumber   |P(z)|  &   =  & \left|\int_K K_{N,\theta}(z,w)P(w)d\mu(w) \right| \, \leq \int_K |K_{N,\theta}(z,w)|d\mu(w)\\
 \nonumber             &\leq & \int_K (K_{N,\theta}(z,z))^{\frac{1}{2}}(K_{N,\theta}(w,w))^{\frac{1}{2}}d\mu(w)=(K_{N,\theta}(z,z))^{\frac{1}{2}}\|(K_{N,\theta}(w,w))^{\frac{1}{2}}\|_{L^1(\mu)}\\
 \nonumber             &\leq &
 (K_{N,\theta}(z,z))^{\frac{1}{2}}\|1\|_{L^2(\mu)}\|(K_{N,\theta}(w,w))\|_{L^2(\mu)}=(K_{N,\theta}(z,z))^{\frac{1}{2}}d(N,\theta)^{\frac{1}{2}}.
\end{eqnarray}
Taking the supremum  of all $P$ as above we have
$\Phi_{K,\theta,N}(z)\leq
(K_{N,\theta}(z,z))^{\frac{1}{2}}d(N,\theta)^{\frac{1}{2}}$, which gives the
first inequality.

For the second inequality, let  $\{P_j\}$ be an orthonormal basis
of $\pi_{N,\theta}$. Then by the Bernstein-Markov property we have
$\|P_j\|_K\leq Ce^{\eps N}$, hence $|P_j(z)|\leq \|P_j\|_K
\Phi_{K,\theta,N}(z) \leq Ce^{\eps N}\Phi_{K,\theta,N}(z), \text{
for all } P_j$. Thus we have
$$K_{N,\theta}(z,z)=\sum\limits_{j=1}^{d(N,\theta)} |P_j(z)|^2\leq  d(N,\theta) C^2e^{2\eps N}(\Phi_{K,\theta,N}(z))^2.$$
Hence we get the second inequality.
\end{proof}

\begin{lem} Let $0\leq \theta<1$. Let $K$ be a regular compact set in \cd. If $(K,\mu) $ satisfies the Bernstein-Markov property for
$\theta$-incomplete polynomials, then we have $$\frac{1}{2N}\log
K_{N,\theta}(z,z)\rightarrow V_{K,\theta}(z)$$
 uniformly on compact subsets of $\cd\setminus\hat{K}_{\theta}$.
\end{lem}
\begin{proof}
We remark that $d(N,\theta)\leq d(N):=d(N,0)$ and $d(N)={ {N+d} \choose d }\leq
(N+d)^d$.

Taking logarithms in \eqref{BSL1}, we obtain
$$-\frac{\log d(N,\theta)}{N} \leq
\frac{\log(\frac{K_{N,\theta}(z,z)}{(\Phi_{K,\theta,N}(z))^2}
)}{N} \leq \frac{\log(Ce^{\epsilon N} d(N,\theta) )}{N}.$$
By the above observation we get
$$-\frac{d}{N}\log (N+d) \leq
\frac{1}{N}\log(\frac{K_{N,\theta}(z,z)}{(\Phi_{K,\theta,N}(z))^2} ) \leq \frac{\log
C}{N}+ \epsilon+ \frac{d}{N}\log (N+d).$$ Since $\eps $  is arbitrary we have
$\frac{1}{N}\log(\frac{K_{N,\theta}(z,z)}{(\Phi_{K,\theta,N}(z))^2} ) \to 0$, which
gives the desired result by Theorem \ref{BSuniform}.
\end{proof}

Let $K$ be  a compact set with admissible weight $w$ on $K$.  Let $\mu$ be a Borel
probability measure on $K$. We say the triple $(K,\mu,w)$   satisfies a
\textbf{weighted Bernstein-Markov property} if there exists $M_N>0$  with
$(M_N)^{1/N}\to~ 1 $ such that for any polynomial $P_N$  of degree $N$,
\begin{equation}\label{WeightedBM}
  \|w^NP_N\|_K \leq M_N \|w^NP_N\|_{L^2(\mu)}.
\end{equation}

We remark that if $K$ is locally regular and $Q$ is continuous then $(K,
(dd^cV_{K,Q})^d,w)$ satisfies a  weighted Bernstein-Markov property by Corollary 3.1 of \cite{BloomWeightedPolynomialsWeightedPluripotentialTheory}. Also
 \\$(D_\phi, (dd^cV_\phi)^d, e^{-\phi}) $ satisfies the weighted Bernstein-Markov
property if $\phi$ is continuous by Theorem 4.5 of \cite{BermanBoucksom}.

\begin{thm}\label{Asymptotic}Let $K$ be  a   compact set with a continuous admissible weight $w$ on $K$. Let $\mu$
be a probability measure on $K$ such that $(K,\mu,w)$ satisfies a weighted
Bernstein-Markov property. Then   we have
\begin{equation}
 \lim_{N\to\infty} \sup_{k=1, \dots,d(N)} (|B_{k,N}(z)|)^{1/N}=
 e^{V_{K,Q}(z)},
\end{equation}where $\{B_{k,N}\}_{k=1}^{d(N)}$ is an orthonormal basis for
the polynomials with degree  at most $N$ with respect to the measure $w^{2N}\mu
$.
\end{thm}
We remark that unlike the unweighted case, where $w=1$, each time $N$ changes the
basis and the $L^2$ norms change.

\begin{proof}
 By the weighted Bernstein-Markov property we have
 $$ \|w^NB_{k,N}\|_K \leq M_N \|w^NB_{k,N}\|_{L^2(\mu)},$$
so we get
$$ \frac{1}{N}\log\frac{|B_{k,N}(z)|}{M_N}\leq Q(z) \text{ on } K.$$Hence
$$ \frac{1}{N}\log\frac{|B_{k,N}(z)|}{M_N}\leq V_{K,Q}(z) \text{ on } \cd .$$
Since $(M_N)^{1/N}\to 1$, we have $\limsup\limits_{N\to\infty}(\sup\limits_{k=1, \dots,d(N)}
(|B_{k,n}(z)|)^{1/N}\leq \limsup\limits_{N\to\infty}(e^{V_{K,Q}(z)}
M_N^{\frac{1}{N}})\leq e^{V_{K,Q}(z)}$.

Now we want to show that $\liminf\limits_{N\to\infty}(\sup_{k=1}^{d(N)}
(|B_{k,N}(z)|)^{1/N}\geq e^{V_{K,Q}(z)}$, for $V_{K,Q}(z) >0$.

Let $P $ be a polynomial of degree at most  $N$ such that $\|w^N
P\|_{K} \leq 1$. We will write $w= e^{-Q}$. Since
$\{B_{k,N}\}_{k=1}^{d(N)}$ is an orthonormal basis we have
$$P(z)=\sum\limits_{j=1}^{d(N)}\left( \int_K P\bar{B}_{j,N}e^{-2NQ}d\mu
\right)B_{j,N}(z).$$ By the triangle inequality we have
$$|P(z)|\leq\sum\limits_{j=1}^{d(N)}\left| \int_K P\bar{B}_{j,N}e^{-2NQ}d\mu
\right| |B_{j,N}(z)|.$$
By the Cauchy-Schwarz inequality we have
$$|P(z)|\leq\sum\limits_{j=1}^{d(N)}\left| \left(\int_K |P|^2e^{-2NQ}d\mu\right)^{1/2} \left(\int_K|{B}_{j,N}|^2e^{-2NQ}d\mu
\right)^{1/2}\right| | B_{j,N}(z)|.$$Now since $\|w^NP\|_{K} \leq 1$ and
$\{B_{k,N}\}_{k=1}^{d(N)}$ is an orthonormal basis we get
$$|P(z)|\leq \sum\limits_{j=1}^{d(N)} | B_{j,N}(z)|.$$
This implies that
\begin{equation}
 |P(z)|
\leq (d(N)) \sup_{j=1}^{d(N)}
(|B_{j,N}(z)|)\text{ for any } z\in \cd.
\end{equation}
We fix   $z\in \cd$.
Then we have
\begin{equation}\label{Esas}e^{V_{K,Q}(z)}\leq\liminf_{N\to\infty}\left(
\sup_{P\in
\pi_{N,0},\, \|w^NP\|_{K} \leq 1}
|P(z)|^{1/N}\right) \leq \liminf_{N\to\infty}(d(N))^{1/N}\left(\sup_{j=1}^{d(N)}
(|B_{j,N}(z)|)\right)^{1/N}.\end{equation}
Here $e^{V_{K,Q}}\leq\liminf\limits_{N\to\infty}\left( \sup\limits_{P\in
\pi_{N,0},\, \|w^NP\|_{K} \leq 1}|P(z)|^{1/N}\right)$ follows from \eqref{REcovertheta=0}. Now
since $(d(N))^{1/N} \to 1$ we get the result.
\end{proof}

\begin{cor}Let $\phi $ be a globally defined continuous admissible weight and  $\mu$ be a Borel probability measure on
$D_\phi$ such that $(D_\phi,\mu,e^{-\phi})$  satisfies a weighted Bernstein-Markov property. Then we have
\begin{equation}
 \lim_{N\to\infty} \sup_{k=1, \dots,d(N)} (|B_{k,N}(z)|)^{1/N}=
 e^{V_{\phi}(z)}.
\end{equation}
Here $\{B_{k,N}\}_{k=1}^{d(N)}$ is an orthonormal basis for the polynomials with
degree  at most $N$ with respect to the measure $e^{-2N\phi}\mu $.
\end{cor}

If  \eqref{WeightedBM} holds for any $P_N \in \pi_{N,\theta}$ then we say
$(K,\mu,w)$ satisfies a \textbf{weighted Bernstein-Markov property for
$\theta$-incomplete polynomials.}

We remark that if a triple $(K,\mu, w)$  satisfies a weighted Bernstein-Markov property, then  it satisfies the weighted Bernstein-Markov property for
$\theta$-incomplete polynomials.

Using only the orthonormal basis for
$\pi_{N,\theta}$ and using Theorem \ref{proof1st} instead of
\eqref{REcovertheta=0} we get the
following theorem by the same proof as for Theorem \ref{Asymptotic}.
\begin{thm}Let $0\leq \theta<1$.
Let $K$ be  a   compact set with   a continuous admissible weight $w$ on $K$. Let $\mu$ be a measure on $K$ such that
$(K,\mu,w)$ satisfies the weighted Bernstein-Markov property for $\theta$-incomplete polynomials. Then  we have
\begin{equation}
 \lim_{N\to\infty} \sup_{k=1, \dots,d(N,\theta)} (|B^\theta_{k,N}(z)|)^{1/N}=
 e^{V_{K,Q,\theta}(z)},
\end{equation}
where  $\{B^\theta_{k,N}\}_{k=1}^{d(N,\theta)}$ is an orthonormal basis for $\pi_{N,\theta}$  with respect to the measure $w^{2N}\mu $.
\end{thm}
\begin{cor}Let $0\leq \theta<1$. Let $\phi $ be a globally defined continuous admissible weight. If $(D_\phi,\mu,e^{-\phi})$  satisfies a weighted Bernstein-Markov property then we have
\begin{equation}
 \lim_{N\to\infty} \sup_{k=1, \dots,d(N,\theta)} (|B^\theta_{k,N}(z)|)^{1/N}=
 e^{V_{\phi,\theta}(z)},
\end{equation}
where $\{B^\theta_{k,N}\}_{k=1}^{d(N,\theta)}$ is an orthonormal basis for $\pi_{N,\theta}$  with respect to the measure $e^{-2N\phi}\mu $.
\end{cor}


Finally, we  prove the strong Bergman asymptotics  in the weighted $\theta$-incomplete   setting following \cite{BermanCn}
closely. We fix $0\leq \theta<1$. Let $\phi $ be a globally defined admissible weight and $ \phi(z) \geq
(1+\eps)\log|z|$ if $|z| \gg 1$.  Let $\{p_1,\dots, p_{d(N,\theta)}\}$ be an orthonormal basis for $\pi_{N,\theta}$ with
respect to the inner product $\langle f, g\rangle := \int_{\cd} f \bar{g} e^{-2N\phi}\omega_d$ where $\omega_d(z)=(
dd^c|z|^2)^d/4^d d!$ on \cd. We denote the $L^2-$norm by
$||p_N||_{N\phi}^2:=||p_N||_{\omega_d,N\phi}^2=\int_{\cd} |p_N(z)|^2 e^{-2N\phi(z)}\omega_d(z)$. We define the
$N-$th $\theta$-incomplete  Bergman function by
\begin{equation}
K_N(z):=K_{N,\theta}^{\phi}(z,z)=\sum_{j=1}^{d(N,\theta)} |p_j(z)|^2e^{-2N\phi(z)}.
\end{equation}
By the reproducing property of the Bergman functions we have

\begin{equation}\label{ExtremalBergman}
K_N(z)=\sup_{p_N\in\pi_{N,\theta}\setminus \{0\}} |p_N(z)|^2e^{-2N\phi(z)}/||p_N||_{N\phi}^2 .
\end{equation}

\begin{thm}\label{SonTheorem} Let $\phi\in C^{2}(\cd)$ with $\phi (z) \geq (1+\epsilon)\log |z|$ for $|z| \gg 1$.
If $V_{\phi,\theta}\in C^{1,1}(\cd \setminus \{0\})$ then
$(dd^c V_{\phi,\theta})^d$ is absolutely continuous with respect to Lebesgue measure on $\cd\setminus\{0\}$ and  $\det (dd^c\phi)\omega_d=
(dd^c V_{\phi,\theta})^d$ on $\cd\setminus \{0\}$ as $(d,d)$ forms with $L^\infty_{loc}(\cd)$ coefficients. For a compact set
$K$ we have a local bound
\begin{equation}
\label{localbound} \frac{1}{d(N,\theta)}K_N(z)\leq C=C(K) \ \hbox{for} \ z\in K.
\end{equation}
Moreover we have
\begin{equation}\label{morse2}\frac{1}{d(N,\theta)}K_N\to \frac{1}{(1-\theta^d)}\chi_{D_{\phi,\theta}\cap P} \frac{\det (dd^c\phi)}{(2\pi)^d} \ \hbox{in} \  L^1(\cd)\end{equation}
and
\begin{equation}\quad \frac{1}{d(N,\theta)}K_N\omega_d \to\frac{1}{(1-\theta^d)} \frac{(dd^c V_{\phi,\theta})^d }{(2\pi)^d}\ \hbox{weak}-* \text{ on } D_{\phi,\theta}\cap P.\end{equation}
\end{thm}

Here $\det (dd^c u) :=\frac{(dd^c  u)^d}{\omega_d}$ and for a twice continuously differentiable function $u$ we have
$\det (dd^c u) =2 i \det [ \frac{\p^2 u}{\p z_j \p \bar{z}_k}]_{j,k=1,\dots ,d}$. The characteristic function of a set $A$
is denoted by $\chi_A$. We remark that we assume  $V_{\phi,\theta}\in C^{1,1}(\cd \setminus\{0\})$.

We will use  the following lemma from measure theory in the proof of the theorem.
\begin{lem}\label{MeasureLemma}\cite[Lemma 2.2]{BermanToeplitz}
Let $(X,\mu)$ be a measure space and let $\{f_N\}$ be a sequence of uniformly bounded, integrable functions on $X$. If
$f$ is a bounded, integrable function on $X$ with
\begin{enumerate}
\item $\lim_{N\to \infty} \int_X f_N d\mu =  \int_X f d\mu$ and
\item $\limsup_{N\to \infty} f_N \leq f$ a.e. with respect to $\mu$
\end{enumerate}
then $f_N$  converges to  $f$ in $L^1(X,\mu)$.
\end{lem}

\begin{proof}[Proof of Theorem \ref{SonTheorem}]
The  $\theta=0$ case is proven by Berman in  \cite{BermanCn}, so we assume  $0<\theta<1$.

By assumption $V_{\phi,\theta}=\phi$ on  $D_{\phi,\theta}\cap P$ and both are $C^{1,1}$ on $D_{\phi,\theta}\cap P$.
Therefore  $\det (dd^c\phi)\omega_d= (dd^c V_{\phi,\theta})^d$ on $D_{\phi,\theta}\cap P$ almost everywhere as $(d,d)$ forms with $L^\infty$ coefficients by the
argument  in Section 12 of \cite{DemaillyPotentialTheoryinSeveralComplexVariables}.

First of all using \eqref{ExtremalBergman} to prove an asymptotic upper bound on $\frac{1}{d(N,\theta)}K_N(z)$ at a point $z_0=(z_1^0,\dots, z_d^0) $, we can assume that  near $z_0$, $\phi$ is of the form
\begin{equation}\phi(z) =\sum_{j=1}^d\lambda_j |z_j-z_j^0|^2 +0(|z-z_0|^3)\end{equation}
as in \cite{BermanCn}. Namely we assume  that $\phi(z_0)=0$ and the first order partial derivatives of $\phi$ vanish at ~$z_0$.

Following \cite{BermanCn}, we have  for each $z_0\in \cd$ there exist $R>0$ and a constant $C$ such that
\begin{equation}
\label{FirstBound}
|\phi(z)|\leq C|z-z_0|^2 \ \hbox{on} \ B(z_0,R),
\end{equation}
and for any $R>0$ we have
\begin{equation}
\label{SecondBound} \lim_{N\to \infty} \left[\sup_{z\in B(0,R)} \left|N\phi(z/\sqrt
N+z_0)-\sum_{j=1}^d\lambda_j|z_j|^2\right|\right]=0.
\end{equation}

We fix $z_0$ be a point in \cd. We take a %
polynomial $p_N\in \pi_{N,\theta}$ satisfying the extremal property \eqref{ExtremalBergman} at $z_0$. Then we have
$$\frac{1}{d(N,\theta)}K_N(z_0) = \frac{|p_N(z_0)|^2e^{-2N\phi(z_0)}}{{d(N,\theta)}||p_N||_{N\phi}^2}=
\frac{|p_N(z_0)|^2}{{d(N,\theta)}\int_{\cd}
|p_N(z)|^2 e^{-2N\phi(z)}\omega_d(z)}$$
By positivity of the integrand we have
$$ \frac{1}{d(N,\theta)}K_N(z_0)\leq  \frac{|p_N(z_0)|^2}{{d(N,\theta)}\int_{|z-z_0|\leq R/\sqrt N} |p_N(z)|^2 e^{-2N\phi(z)}\omega_d(z)}.%
$$
We choose $R$ as in (\ref{FirstBound}) so that we can replace $\phi(z)$ by $C|z-z_0|^2$ in the integrand and thus we have
$$\frac{1}{d(N,\theta)}K_N(z_0)\leq \frac{|p_N(z_0)|^2}{d(N,\theta)\int_{|z-z_0|\leq R/\sqrt N} |p_N(z)|^2 e^{-2NC|z-z_0|^2}\omega_d(z)}.$$
We apply the subaveraging property to the subharmonic function $|p_N|^2$ on the ball $\{|z-z_0|\leq R/\sqrt N\}$
with respect to the radial  probability measure with center $z_0$ $\frac{e^{-2NC|z-z_0|^2}\omega_d(z)}{\int_{|z-z_0|\leq R/\sqrt N}
e^{-2NC|z-z_0|^2}\omega_d(z)}$ to obtain
\begin{eqnarray}
\nonumber  \frac{1}{d(N,\theta)}K_N(z_0) &\leq& \frac{1}{d(N,\theta)\int\limits_{|z-z_0|\leq R/\sqrt N} e^{-2NC|z-z_0|^2}\omega_d(z)} \\
\nonumber   &\leq& \frac{N^d}{d(N,\theta)\int_{|z'|\leq R}  e^{-2C|z'|^2}\omega_d(z')}
\end{eqnarray}
For the  last inequality we used  a change of variable $ z\to z':=(z-z_0)\sqrt N$, where $\omega_d(z') = N^d\omega_d(z)$.
Since  $d(N,\theta) \asymp (1-\theta^d) d(N,0)$, we have  $ d(N,\theta) \geq (1-\tilde{\theta}^d) d(N,0)$ for all $N\geq N_0$ for some   $\widetilde{\theta}\geq \theta$. Now using the estimate $d(N,\theta) \geq (1-\tilde{\theta}^d) d(N,0)=(1-\tilde{\theta}^d){d+N\choose d}\geq
(1-\tilde{\theta}^d)N^d/d!$ for all $N\geq N_0$, we get
$$\frac{1}{d(N,\theta)}K_N(z_0) \leq  \frac{d!}{(1-\tilde{\theta}^d) \int_{|z'|\leq R}  e^{-2C|z'|^2}\omega_d(z')}\text{ for all } N\geq N_0 .$$
The right hand side of the inequality is uniformly bounded.
As $z_0$ varies on the compact set $K$, we get  a constant $C(K)$ giving a
local bound  for all $N\geq N_0$.    By continuity of $\frac{1}{d(N,\theta)}K_N(z)$, and considering the $\max_{N=1,\cdots,N_0}\sup_{z\in K}\frac{1}{d(N,\theta)}K_N(z)$  we  get the local bound
\eqref{localbound} holds at each point of $K$.

For the rest of the proof, we fix $z_0$ and start with the inequality
$$\frac{1}{d(N,\theta)}K_N(z_0) \leq \frac{|p_N(z_0)|^2}{{d(N,\theta)} \int_{|z-z_0|\leq R/\sqrt N} |p_N(z)|^2 e^{-2N\phi(z)}\omega_d(z)}$$
which holds for any $R>0$. By using the same change of variable and estimates as above  we get
$$ \frac{1}{d(N,\theta)}K_N(z_0)  \leq \frac{d!|p_N(z_0)|^2
  }{(1-\tilde{\theta}^d)\int_{|z'|\leq R} |p_N(z'/\sqrt N +z_0)|^2 e^{-2N\phi(z'/\sqrt N
  +z_0)}\omega_d(z')},$$ for all $N\geq N_0$ where $\tilde{\theta}\geq \theta$.
Multiplying the integrand by $e^{-2\sum_{j=1}^d\lambda_j|z'_j|^2} e^{2\sum_{j=1}^d\lambda_j|z'_j|^2}$ and taking the infimum of
$\exp \left[-2 \left|N\phi(z'/\sqrt N)-\sum_{j=1}^d \lambda_j|z'_j|^2\right|\right]$ on $B(0,R)$ out of the integral, we get
$$  \frac{1}{d(N,\theta)}K_N(z_0) \leq   \frac{d!|p_N(z_0)|^2\exp \left[2\sup_{|z'|\leq R} \left|N\phi(z'/\sqrt N)-\sum_{j=1}^d
\lambda_j|z'_j|^2\right|\right]}{(1-\tilde{\theta}^d)\int_{|z'|\leq R}  |p_N(z'/\sqrt N +z_0)|^2 e^{-2\sum_{j=1}^d\lambda_j|z'_j|^2}
\omega_d(z')},$$ for all $N\geq N_0$.
We apply  the subaveraging property to the subharmonic function $|p_N(z'/\sqrt N +z_0)|^2$  with respect to radial
probability measure $\frac{e^{-2\sum_{j=1}^d\lambda_j|z'_j|^2}\omega_d(z')}{\int_{|z'|\leq R}
e^{-2\sum_{j=1}^d\lambda_j|z'_j|^2}\omega_d(z')}$ and we get
$$\frac{1}{d(N,\theta)}K_N(z_0) \leq  \frac{d!\exp \left[2\sup_{|z'|\leq R} \left|N\phi(z'/\sqrt N)-\sum_{j=1}^d
\lambda_j|z'_j|^2\right|\right]}{(1-\tilde{\theta}^d)\int_{|z'|\leq R}  e^{-2\sum_{j=1}^d\lambda_j|z'_j|^2}\omega_d(z')}$$ for all $N\geq N_0$.
By   \eqref{SecondBound},   $\exp \left[2\sup_{|z'|\leq
R} \left|N\phi(z'/\sqrt N)-\sum_{j=1}^d\lambda_j|z'_j|^2\right|\right] \to 1$ as $N\to \infty$. Therefore we have
$$  \limsup_{N\to \infty} \frac{1}{d(N,\theta)}K_N(z_0)\leq \frac{d!}{{(1-\tilde{\theta}^d)\int_{|z'|\leq R}  e^{-2\sum_{j=1}^d\lambda_j|z'_j|^2}\omega_d(z')}}.$$
As $R\to \infty$ the Gaussian integral on the right hand side goes to $\frac{\pi^d}{ 2^d\lambda_1 \cdots \lambda_d}$ if
all $\lambda_j >0$ and to $+\infty$ otherwise. Since $\det (dd^c\phi(z_0))= 4^d d! \lambda_1 \cdots \lambda_d$ we
have
\begin{equation}
\limsup_{N\to \infty} \frac{1}{d(N,\theta)}K_N(z)\leq \frac{1}{(1-\tilde{\theta}^d)}\chi_{{D_{\phi,\theta}}\cap
P}\frac{\det (dd^c\phi)}{(2\pi)^d} \text{ a.e on } \cd.
\end{equation}
Letting $\tilde{\theta} \to \theta$ we obtain
\begin{equation}
\label{morse} \limsup_{N\to \infty} \frac{1}{d(N,\theta)}K_N(z)\leq \frac{1}{(1-\theta^d)}\chi_{{D_{\phi,\theta}}\cap
P}\frac{\det (dd^c\phi)}{(2\pi)^d} \text{ a.e on } \cd.
\end{equation}
By the definition of $\limsup$  and using the extremal property \eqref{ExtremalBergman}, we get
\begin{equation}
\label{oldmorse2} \frac{1}{N^d}|p_N(z)|^2e^{-2N\phi(z)}/||p_N||_{N\phi}^2\leq C_N  \text{ on } D_{\phi,\theta} \text{
for any  } p_N\in \mathcal \pi_{N,\theta},
\end{equation}
where $C_N = \frac{1}{(1-\theta^d)}\sup_{z\in D_{\phi,\theta}\cap P}\frac{\det (dd^c\phi(z))}{(2\pi)^d}$. Next we will
show that
\begin{equation}
\label{DominationMorse}
 \frac{1}{N^d}K_N(z)\leq C_Ne^{-2N(\phi(z)-V_{\phi,\theta}(z))} \ \hbox{on} \ \cd .
\end{equation}
Let $p_N\in \pi_{N,\theta}$ such that $\|p_N\|^2_{N\phi}=N^{-d}$, then  by \eqref{oldmorse2} we have
$$  |p_N(z)|^2e^{-2N\phi(z)}\leq C_N \ \hbox{on} \ D_{\phi,\theta}.$$
By taking logarithms we get
$$\frac{1}{2N}\log |p_N(z)|^2\leq \phi(z)+\frac{1}{2N}\log C_N  \ \hbox{on} \ D_{\phi,\theta}$$
and thus we have
$$\frac{1}{2N}\log |p_N(z)|^2\leq V_{\phi,\theta}(z)+\frac{1}{2N}\log C_N  \ \hbox{on} \ \cd .$$
So from the extremal property of Bergman functions \eqref{ExtremalBergman} we obtain
$$ \frac{1}{N^d}K_N(z)=\sup_{||p_N||^2_{N\phi}=N^{-d}} |p_N(z)|^2e^{-2N\phi(z)}\leq C_Ne^{-2N(\phi(z)-V_{\phi,\theta}(z))} \ \hbox{ on}  \ \cd.$$
Since $\phi(z)>V_{\phi,\theta}(z)$ on $\cd \setminus D_{\phi,\theta}$, we obtain
$$\lim_{N\to\infty}\frac{1}{N^d}K_N(z)=0 \hbox{ on }   \cd \setminus D_{\phi,\theta}.$$
Using $d(N,\theta)\asymp(1-\theta^d)d(N,0)$, we obtain
$$\lim_{N\to\infty}\frac{1}{d(N,\theta)}K_N(z)=0 \hbox{ on }   \cd \setminus D_{\phi,\theta}.$$
From (\ref{DominationMorse}) and the growth assumption on $\phi$, for a sufficiently large $R$, there is a $C$ with
\begin{equation}\label{ForLebesgue}
\frac{1}{N^d}K_N(z)\leq C|z|^{-2N\epsilon} \ \hbox{for} \ |z|>R.%
\end{equation}
By combining the local bound \eqref{localbound} and  above estimate \eqref{ForLebesgue} we get a global bound for  $\frac{1}{d(N,\theta)}K_N$.
Therefore  Lebesgue's  dominated convergence theorem gives that
\begin{equation}\label{offd}
\lim_{N \to \infty}\int_{\cd \setminus D_{\phi,\theta}}\frac{1}{d(N,\theta)}K_N\omega_d =0.
\end{equation}
Next we show that
\begin{equation}
\label{IntegralEq} \lim_{N\to \infty}\int_{D_{\phi,\theta}\cap P}\frac{1}{d(N,\theta)}K_N\omega_d
=\frac{1}{(1-\theta^d)}\int_{D_{\phi,\theta}\cap P} \frac{\det ( dd^c\phi)}{(2\pi)^d}\omega_d.
\end{equation}
To prove (\ref{IntegralEq}), we know that
$$\int_{\cd}K_N\omega_d = d(N,\theta) $$
and using (\ref{offd}) we have
$$1 = \lim_{N\to \infty}\int_{\cd}\frac{1}{d(N,\theta)}K_N\omega_d= \lim_{N\to \infty}\int_{D_{\phi,\theta}\cap P}\frac{1}{d(N,\theta)}K_N\omega_d.$$
On the other hand, using the positivity of the integrand and applying \eqref{morse} on  $D_{\phi,\theta}$, we have
$$1 = \lim_{N\to \infty}\int_{D_{\phi,\theta}}\frac{1}{d(N,\theta)}K_N\omega_d\leq \frac{1}{(1-\theta^d)}\int_{D_{\phi,\theta}\cap P}\frac{ \det (dd^c\phi) }{(2\pi)^d}\omega_d.$$
By  the first part of this theorem, we can replace $\det (dd^c\phi)\omega_d$ by $(dd^c V_{\phi,\theta})^d$  which has
total mass ${(2\pi)^d}{(1-\theta^d)}$ on $D_{\phi,\theta}\cap P$, hence we have
$$ 1 = \lim_{N\to \infty}\int_{D_{\phi,\theta}\cap P}\frac{1}{d(N,\theta)}K_N\omega_d\leq \frac{1}{(1-\theta^d)}\int_{D_{\phi,\theta}\cap P}\frac{(dd^c V_{\phi,\theta})^d}{(2\pi)^d} =\frac{{(2\pi)^d}{(1-\theta^d)}}{(2\pi)^d(1-\theta^d) }=1.$$
This gives (\ref{IntegralEq}). We will use this relation, together with \eqref{morse2}, to show that
$$\frac{1}{d(N,\theta)}K_N \to \frac{1}{(1-\theta^d)}\chi_{D_{\phi,\theta}\cap P}\frac{\det (dd^c\phi)}{(2\pi)^d} \
\hbox{in} \ L^1(\cd).$$

We set $f_N := \frac{1}{d(N,\theta)}K_N$ and $f:= \frac{1}{(1-\theta^d) }\chi_{D_{\phi,\theta}\cap P}\frac{\det
(dd^c\phi)}{(2\pi)^d}$. By  the upper bound \eqref{morse} we have $\limsup\limits_{N\to \infty} f_N \leq ~f$  almost everywhere
  and by \eqref{offd}  and \eqref{IntegralEq}
 we have $\lim_{N\to \infty} \int_{\cd} f_N \omega_d =  \int_{\cd}f \omega_d$. Thus by Lemma~ \ref{MeasureLemma}  we get the convergence of
$\frac{1}{d(N,\theta)}K_N $ to $\frac{1}{(1-\theta^d)} \chi_{D_{\phi,\theta}\cap P}\frac{\det (dd^c\phi)}{(2\pi)^d}$ in
$L^1(\cd)$. This implies  the weak-* convergence of $\frac{1}{d(N,\theta)}K_N\omega_d$ to
$\frac{1}{(1-\theta^d)}\chi_{D_{\phi,\theta}\cap P}\frac{\det (dd^c\phi)}{(2\pi)^d}\omega_d$ and completes the proof
of the theorem.
\end{proof}

 \nocite{*}
\bibliographystyle{amsalpha}

\bibliography{xbib}
\end{document}